\documentclass[11pt]{amsart}
\usepackage{enumitem}
\usepackage{amsmath}
\usepackage{amssymb}
\usepackage{amsthm}
\usepackage{verbatim}
\usepackage{stmaryrd}
\usepackage{srcltx}
\usepackage{xcolor}
\usepackage[utf8]{inputenc}

\newtheorem{Theorem}{Theorem}[section]

\newtheorem{Lemma}[Theorem]{Lemma}

\newtheorem{Remark}[Theorem]{Remark}
\newtheorem{thm}{Theorem}[section]
\newtheorem{cor}[thm]{Corollary}

\newtheorem{lem}[thm]{Lemma}

\theoremstyle{definition}
\newtheorem{defn}[thm]{Definition}

\theoremstyle{remark}

\newtheorem{example}[thm]{Example}

\title[Limiting Carleman Weights and CTA Manifolds]
{Limiting Carleman Weights and Conformally Transversally Anisotropic Manifolds}

\def\R{\mathbb{R}}

\def\2L{\Lambda_{\tilde{\gamma}}}
\def\1L{\Lambda_{\gamma}}

\DeclareMathOperator{\grad}{grad}

\newcommand{\lieg}{\mathfrak{g}}

\renewcommand{\leq}{\leqslant}
\renewcommand{\geq}{\geqslant}

\DeclareMathOperator{\Ric}{Ric}

\newcommand{\spann}[1]{\mathrm{span}{#1}}

\def \dim{{\mbox {dim}}\,}

\newcommand{\mR}{\mathbb{R}}

\newcommand{\abs}[1]{\lvert #1 \rvert}

\begin{document}

\author{Pablo Angulo}
\address{Department of Mathematics,
ETS de Ingenieros Navales,
Universidad Polit\'ecnica de Madrid}
\curraddr{}
\email{pablo.angulo@upm.es}

\author{Daniel Faraco}
\address{ Department of Mathematics, Universidad Aut\'onoma de Madrid, and ICMAT CSIC-UAM-UCM-UC3M}
\curraddr{}
\email{daniel.faraco@uam.es}

\author{Luis Guijarro}
\address{ Department of Mathematics, Universidad Aut\'onoma de Madrid, and ICMAT CSIC-UAM-UCM-UC3M}
\curraddr{}
\email{luis.guijarro@uam.es}

\author[M. Salo]{Mikko Salo}
\address{University of Jyvaskyla, Department of Mathematics and Statistics, PO Box 35, 40014 University of Jyvaskyla, Finland}
\email{mikko.j.salo@jyu.fi}

\thanks{The first three authors were supported by research grants  MTM2014-57769-1-P, MTM2014-57769-3-P and
MTM2017-85934-C3
from the Ministerio de Ciencia e Innovaci\'on (MCINN), by ICMAT Severo Ochoa projects SEV-2011-0087 and  SEV-2015-0554 (MINECO), and by the ERC 301179. The fourth author was supported by the Academy of Finland (grants 284715 and 309963) and by ERC under Horizon 2020 (ERC CoG 770924).
}

\begin{abstract}
We analyze the structure of the set of limiting Carleman weights in all conformally flat manifolds, $3$-manifolds, and $4$-manifolds. In particular
we give a new proof of the classification of Euclidean limiting Carleman weights, and show that there are only three basic such weights up to the action of the conformal group. In dimension three we show that if the manifold is not conformally flat, there could be one or two limiting Carleman weights.
We also characterize the metrics that have more than one limiting Carleman weight.
In dimension four we obtain a complete spectrum of examples according to the structure of the Weyl tensor. In particular, we construct unimodular Lie groups whose Weyl or Cotton-York tensors have the symmetries of conformally transversally anisotropic manifolds, but which do not admit limiting Carleman weights.
\end{abstract}


\maketitle
\pagestyle{myheadings}
\markleft{P.ANGULO, D. FARACO, L. GUIJARRO,  AND  M.SALO}

\section{Introduction}

\subsection{Background}
A version of the  inverse problem of Calder\'on \cite{Calderon} asks for the determination
 of a potential $q$ from boundary measurements (given by the Dirichlet-to-Neumann map $\Lambda_q$) for the Schr\"odinger operator $-\Delta_g+q$ in a compact Riemannian manifold $(M,g)$ with boundary. There is an extensive literature for the case where $(M,g)$ is a domain in Euclidean space (see the survey \cite{Uhlmann_survey}). The corresponding problem for a compact manifold $(M,g)$ has been solved in two dimensions \cite{GuillarmouTzou}, or when the underlying structures are real-analytic (see \cite{LeeUhlmann, LassasUhlmann, GuillarmouSaBarreto, LassasLiimatainenSalo}). The problem remains open in general when $\dim(M) \geq 3$.
 
There has been progress in the above problem when $(M,g)$ belongs to the class of \emph{conformally transversally anisotropic} (CTA) manifolds . A manifold is said to be CTA if
there exists a conformal factor $c(x)$ and an $(n-1)$-manifold
$(M_0,g_0)$ such that $(M,g)$ is isometric to a domain in $\mathbb{R} \times M_0$ with the metric $c(x)(dx_1^2 \oplus g_0)$. In this article will always assume that 
\[
n = \dim(M) \geq 3.
\]
If $(M,g)$ is a CTA manifold, it turns out that one can construct complex geometrical optics solutions for the equation $(-\Delta_g + q)u = 0$ in $M$ as in the classical approach of \cite{SylvesterUhlmann} in the Euclidean case. This is based on the fact that the function $\varphi(x)=x_1$ is a so called limiting Carleman weight (LCW), see \cite{KSU, DKSaU}. It was proved in \cite{DKSaU} that the existence of an LCW is locally equivalent to the manifold begin CTA. Moreover, if $(M,g)$ is a CTA manifold, one can solve the inverse problem of determining the potential $q$ from boundary measurements if additionally 
\begin{itemize}
\item $(M_0,g_0)$ is simple \cite{DKSaU};
\item $(M_0,g_0)$ has injective geodesic X-ray transform \cite{DKLS}; or
\item $(M,g)$ is conformal to a subdomain of $\mR^2 \times M_0$, instead of just $\mR \times M_0$ (this follows by combining \cite{DKLS} with \cite[Theorem 1.3]{Salo_normalform}).
\end{itemize}

Unfortunately, due to the conformal invariance it is not always easy to decide whether a manifold is CTA. In \cite{AFGR,AFG} conformal symmetries were exploited to investigate whether  a manifold is CTA or not. In particular, \cite{AFGR} gave necessary conditions for a manifold to be CTA by proving that the Weyl or Cotton tensor of a CTA manifold needs to have certain specific algebraic structure. This work also gave the first explicit examples of manifolds (e.g.\ $\mathbb{C} P^2$, $\mathrm{Nil}$, $\widetilde{\mathrm{SL_2(\mR)}}$) that do not admit LCWs; it had been shown earlier that generic manifolds do not admit LCWs \cite{LiimatainenSalo, An}.  The article \cite{AFG} complemented \cite{AFGR} by giving several sufficient conditions for a manifold to admit an LCW in dimensions $3$ and $4$.



The works \cite{AFGR, AFG} were concerned with necessary and sufficient conditions for a manifold to admit at least one LCW. The first goal of the present article is to analyze and classify the set of all possible LCWs in a given Riemannian manifold. In particular we prove that having several LCWs imposes very strong symmetries in the manifold. The second aim is to further clarify the difference between the necessary conditions in \cite{AFGR} and the presence of an actual LCW. 

For the first goal we start by revisiting the case where LCWs are most abundant, i.e.\ Euclidean space.

\subsection{LCWs in Euclidean space}

LCWs in a domain $\Omega \subset \mR^n$, $n \geq 3$, were characterized in \cite[Theorem 1.3]{DKSaU} where it was proved that any LCW in $\Omega$ belongs to one of six different families up to translation and scaling. The proof was based on two facts:
\begin{itemize} 
\item the level sets of an LCW are umbilical hypersurfaces; and 
\item any umbilical hypersurface in $\mR^n$ is part of a hyperplane or sphere.
\end{itemize}
The result then followed by an ODE analysis of the parameters that define the level sets.

We begin by giving a new proof of the classification of LCWs in $\mR^n$. Instead of using the fact that the level sets of LCWs are umbilical hypersurfaces, we will start from the observation (also made in \cite{DKSaU}) that any LCW $\varphi$ has an associated conformal Killing vector field $\abs{\nabla \varphi}^{-2} \varphi$. Thus one could try to classify LCWs by first classifying all conformal Killing vector fields, and then checking which conformal Killing vector fields give rise to LCWs.

It is well known (and recalled in Lemma \ref{lemma_ck_euclidean}) that in $\mR^n$ with $n \geq 3$, any conformal Killing vector field is of the form 
\[
X(x) = (\alpha \cdot x)x - \frac{1}{2} \alpha \abs{x}^2 + cx + Bx + \gamma
\]
for some $\alpha, \gamma \in \mR^n$, $c \in \mR$, and some skew-symmetric matrix $B$. We will write $X = (\alpha, c, B, \gamma)$ for short. Using this characterization, the fact that $d(|X|^{-2} X) = 0$ for $X$ arising from an LCW, and arguments based on the conformal invariance of the problem, we obtain the following restatement of \cite[Theorem 1.3]{DKSaU}.

\begin{thm} \label{thm_lcw_rn_classification} 
Let $\Omega \subset \mR^n$, $n \geq 3$, be a connected open set, and let $\varphi$ be an LCW in $\Omega$. Then $\varphi(x) = a \psi(x-x_0) + b$ for some $a \in \mR \setminus \{0\}, b\in\mR$ and $x_0 \in \mR^n$, where $\psi$ is one of the following six functions and $X = |\nabla \psi|^{-2} \nabla \psi$ is the conformal Killing vector field for $\psi$:
\[
\begin{array}{ll}
\psi(x) = \gamma \cdot x, & X = (0, 0, 0 , \gamma), \\[7pt]
\psi(x) = \log |x|, & X = (0, 1, 0, 0), \\[7pt]
\psi(x) = \mathrm{arc\,tan} \,\frac{\gamma \cdot x}{\sigma \cdot x}, & X = (0, 0, \gamma \wedge \sigma, 0), \\[7pt]
\psi(x) =- 2 \frac{\gamma \cdot x}{|x|^2}, & X = (\gamma, 0, 0, 0), \\[7pt]
\psi(x) =  \frac{1}{\sqrt{s}} \,\mathrm{arctan}\,\frac{-2 \gamma \cdot x/\sqrt{s}}{|x|^2/s - 1}, & X = (\gamma, 0, 0, \frac{s}{2} \gamma), \\[7pt]
\psi(x) = \frac{1}{\sqrt{s}} \,\mathrm{arctanh}\,\frac{-2\gamma \cdot x/\sqrt{s}}{|x|^2/s + 1}, & X = (\gamma, 0, 0, -\frac{s}{2} \gamma). \end{array}
\]
Above $\gamma, \,\sigma \in \mR^n$ satisfy $|\gamma| = |\sigma| = 1$ and $\gamma \cdot \sigma = 0$, and $s > 0$.
\end{thm}


The six families above are the same as in \cite{DKSaU}. However, we also prove that the last three families can be obtained from the first three by conformal mappings, thus reducing the number of basic LCWs to three:

\begin{thm}\label{thm_lcw_rn_orbits}
  The group of conformal transformations in $\mR^n$ acts on the set of LCWs by $\psi \mapsto F^\ast\psi=\psi\circ F$.
  This action, combined with the action $\psi \mapsto a\psi + b$ where $a \in \mR \setminus \{0\}$ and $b \in \mR$, has exactly three orbits, given by the following representatives ($e_1$ and $e_2$ are the first two vectors of the canonical basis):
  \[
  \begin{array}{l}
  \psi_1(x) = e_1 \cdot x\\[7pt]
  \psi_2(x) = \log |x|\\[7pt]
  \psi_3(x) = \mathrm{arc\,tan} \,\frac{e_1 \cdot x}{e_2 \cdot x}.
  \end{array}
  \]
  In other words, for any LCW $\varphi$ defined in an open set $\Omega\subset\mR^n$, there is exactly one $i\in\{1,2,3\}$, a conformal transformation $F$ defined on the one point compactification of $\mR^n$, and two real numbers $a, b$ with $a \neq 0$, such that $\varphi$ is the restriction of $a\left( \psi_i\circ F\right) + b$ to $\Omega$.

  If we only consider affine conformal mappings (without the inversion), there are exactly six orbits, corresponding to the six families in Theorem \ref{thm_lcw_rn_classification}.
\end{thm}

\subsection{LCWs on general manifolds}
Next, we turn our attention to general manifolds. We say that LCWs $\varphi_1, \ldots, \varphi_m$ are \emph{orthogonal} if their gradients are orthogonal at each point. The next theorem describes the underlying geometries and shows that orthogonal LCWs can be chosen as coordinates.


\begin{thm}
\label{thm:diagonal_metric}
Let $(M,g)$ be an $n$-dimensional  Riemannian manifold. Let $p\in M$, and suppose that in an open neighbourhood of $p$, the conformal class $[g]$ admits $m$ different orthogonal LCWs  $\varphi_1,\dots,\varphi_m$. Then there are local coordinates $\Phi=(z_1,\dots, z_n)$ near $p$ such that
\begin{enumerate}
\item $z_1=\varphi_1,\dots, z_m=\varphi_m$;
\item some conformal multiple of $g$ has the local expression
\[
\left[\begin{array}{ccccc}
1 & 0 & \dots & 0  & 0\\
0 & f_2(z_{m+1},\dots,z_n) & \dots & 0 & 0\\
 &  & \vdots &  & \\
0 & 0 & \dots & f_m(z_{m+1},\dots,z_n) & 0\\
0 & 0 & \dots & 0 & D(z_{m+1},\dots,z_n)\\
\end{array}\right]
\]
where $D$ is an $(n-m)\times (n-m)$ symmetric matrix.
\end{enumerate}
Conversely, if in some local coordinates $(z_1, \ldots, z_n)$ the metric has the above form up to a conformal multiple, then $z_1, \ldots, z_m$ are orthogonal LCWs.
\end{thm}

The above Theorem has the following relevant corollary:

\begin{cor}
\label{cor:every orth LCW}
    An $n$-dimensional manifold with $n$ orthogonal LCWs is conformally flat.
\end{cor}

As a matter of fact, Theorem \ref{thm:diagonal_metric} in combination with the analysis of \cite{AFG, AFGR}  yields a complete understanding of $3$-manifolds. First we obtain the following theorem:
 
\begin{thm}
\label{thm:3-dim_case_several_LCWs}
Let $(M,g)$ be a $3$-manifold. Locally 
\begin{itemize}
\item[i)]  $(M,g)$ admits a limiting Carleman weight if and only if $(M,g)$ is conformal to $\R \times M_0$, where $M_0$ is a $2$-manifold;
\item[ii)]  $(M,g)$ admits two limiting Carleman weights with linearly independent gradients if and only if $(M,g)$ is conformal to $\R \times S$, where $S$ is a surface of revolution;
\item [iii)] $(M,g)$ admits three limiting Carleman weights with linearly independent gradients if and only if $(M,g)$ is conformally flat.
\end{itemize}
\end{thm}

Note that if $\varphi_1$ and $\varphi_2$ are two LCWs with linearly independent gradients, then $\varphi_1$ is not of the form $a\varphi_2 +b$ for real numbers $a \neq 0, b$.

 For $3$-manifolds, the necessary condition from \cite{AFGR} for a manifold to be CTA
 is $\det(CY)=0$, where $CY$ is the Cotton-York tensor.  It was still open whether this  condition is sufficient.  Unimodular  Lie groups admit left invariant metrics which are easy to work with.
 Equipped with these metrics, they become good candidates for counterexamples as their curvature tensors, 
 and thus the eigenflag directions, are also left invariant.
 Lie brackets can be prescribed so that the left invariant metric has a Cotton tensor of the desired type, but the left invariant distributions spanned by the eigenflag directions are not integrable.
 It is possible to make these heuristics precise and  to find a specific example of a Lie group whose Cotton-York tensor is nowhere zero, has the symmetries that correspond to a CTA manifold, and yet the manifold is not CTA. This is the content of our next theorem.

\begin{thm}
\label{thm:3-dim-unimodular group}
There exists a unimodular three dimensional Lie group $G$ such that $det(CY)=0$ but $G$ is not CTA.
\end{thm}

As discussed in Section \ref{sec6}, the analysis on three manifolds is now completed.  Let us turn to four manifolds. In dimension four  the conformal symmetries are described by the Weyl
 tensor. In \cite{AFGR} we found that the gradients of LCWs need to be related to so called "eigenflag" directions for the Weyl tensor.
 
\begin{defn}[\cite{AFGR}]\label{eigenflag}
 Let $W$ be a \emph{Weyl tensor} in $S^2(\Lambda^2V)$. We say that  $W$ satisfies the \emph{eigenflag condition} if and only if there is a nonzero vector $v\in V$ such that $W(v \wedge v^{\perp})\subset v \wedge v^{\perp}$, where by $v \wedge v^{\perp}$ we denote the set of bivectors
\[
\left\{\,v\wedge w\,:\,w\in V, \left\langle\,v,w\right\rangle =0\,\right\}.
\]
A one-dimensional subspace of $V$ is called an \emph{eigenflag direction} if it is spanned by some $v$ satisfying the above condition.

\end{defn}

In \cite{AFG} we classified the Weyl tensor in types A, B, C, D according to the number of eigenflag directions.
 
\begin{lem}\label{lem:Weyl}
 The algebraic Weyl operators $W$ in a vector space of dimension $4$ fall into one of the following types:

 \begin{description}
  \item[A] $W$ has no eigenflag directions.
  \item[B] $W$ has at least one eigenflag direction and three different eigenspaces of dimension $2$. In this case, $W$ has exactly four eigenflag directions.
  \item[C] $W$ has at least one eigenflag direction and two eigenspaces with dimensions four and two. In this case, the eigenflag directions for $W$ consist of the union of two orthogonal 2-planes.
  \item[D] $W$ is null. All directions are eigenflag.
 \end{description}
\end{lem}

 We also showed  that the
product of two scalene ellipsoids had Weyl tensor of type C but it
 was not CTA, showing that for four manifolds our necessary condition was not sufficient. However, this left the question whether the product structure was important in the counterexample and genuine four dimensional examples exist.
With this aim, we generalize unimodular Lie groups, and equip them with
left invariant metrics, to find examples of Riemannian manifolds whose
Weyl tensors have the symmetries that correspond to a CTA manifold, or
even to a product of surfaces, but still they are not conformal to a
product of lower dimensional manifolds in any way.

For the case of Weyl tensors of type $B$,  there are four eigenflag directions that could admit LCWs; we provide examples of manifolds where only one, two
or three of the eigenflags correspond to actual limiting Carleman weights. 
The case of four LCWs is settled by the following observation that follows immediately from Corollary \ref{cor:every orth LCW} after observing that for Weyl tensors of type B, the eigenflag directions form an orthogonal basis.

\begin{cor}
Let $(M,g)$ be a $4$-manifold with a Weyl tensor of type B which admits four limiting Carleman weights. Then $(M,g)$ is conformally flat.
\end{cor}

It remains to find an example of a metric with a Weyl tensor of type B, but not conformally transversally anisotropic. Inspired by dimension three, we look for them among the significantly more complex structure of 4-dimensional Lie groups.

\begin{thm}
\label{thm:4dim_no_LCW}
There exists a generalized unimodular Lie group $G$ with a left invariant metric with Weyl tensor of type B, but such that $G$ is not conformally anisotropic.
\end{thm}

Thus Weyl tensors of type $B$ are fully understood. 
Finally we discuss the case of Weyl tensors of type $C$. We recall that this is the case appearing when $M$ is the Riemannian  product of surfaces. This was already considered in
\cite{AFG}.

\begin{thm}\label{thm: products of surfaces which admit LCW}

  Let $(S_ 1, g_ 1)$ and $(S_2,g_2)$ be open subsets of $\R^2$ with Riemannian metrics. Assume that the Weyl operator of the product metric does not vanish at any point.
  The following are equivalent:
  \begin{itemize}
    \item $(S_1, g_1)$ is locally isometric to a surface of revolution.

    \item $(S_1	\times S_2, g_1\times g_2)$  admits a LCW that is everywhere tangent to the first
    factor.

  \end{itemize}
\end{thm}

This theorem was behind our example of two scalene ellipsoids having eigenflag directions but not being CTA. 
In Section \ref{subsec:4dim_no_LCWs} we construct a Lie group whose Weyl tensor is type C, has no LCWs, and it is not  a product of surfaces.

To help the reader, we finish the introduction by including  a few tables that summarize some of the results from this paper and from \cite{AFG} in regard to the question of which eigenflag directions can be realized by LCWs.

\begin{center}
\textbf{Dimension 4, Weyl type B:}
\end{center}
Structure of eigenflags: 4 directions, forming an orthogonal basis.

\begin{itemize}
\item Number of LCWs:
\begin{itemize}
\item Zero: example given in Section \ref{subsec:4dim-noLCW};
\item One: $M=\mathbb{R}\times N^3$, where the metric in $N^3$ has no LCWs;
\item Two: $M$ has a warped product structure, Section \ref{subsec:4dim,2LCWs};
\item Three: $M$ is an iterated warped product, Section \ref{subsec:4dim,3LCWs};
\item Four: $M$ is conformally flat, Corollary \ref{cor:every orth LCW}.
\end{itemize}
\end{itemize}

\begin{center}
\textbf{Dimension 4, Weyl type C:}
\end{center}
\begin{itemize}
\item Structure of eigenflags: two orthogonal 2-planes.
\item They do not need to have LCWs: Theorem \ref{thm:4dim_no_LCW}, proof in Section \ref{subsec:4dim_no_LCWs};
\item They do not need to arise from product of surfaces.
\end{itemize}

The rest of this article is organized as follows. Section \ref{sec_conformal_killing} discusses conformal Killing fields, Section \ref{sec3} deals with the action of the conformal group, and Section \ref{sec4} characterizes Euclidean LCWs by proving
Theorems~\ref{thm_lcw_rn_classification} and \ref{thm_lcw_rn_orbits}.  Section \ref{sec5} analyzes manifolds with several  orthogonal LCWs, and 
 Section \ref{sec6} deals with three manifolds. The analysis of four manifolds is given in Section \ref{sec7} (structural results) and Section \ref{sec8} (counterexamples).

\section{Conformal Killing vector fields} \label{sec_conformal_killing}

We begin with a result that follows from \cite[Lemma 2.9]{DKSaU}:

\begin{Lemma} \label{lemma_lcw_dksau}
Let $(M,g)$ be a simply connected open Riemannian manifold, and let $\varphi \in C^{\infty}(M)$ satisfy $d\varphi \neq 0$ everywhere. Then $\varphi$ is an LCW in $(M,g)$ if and only if $\abs{\nabla \varphi}^{-2} \nabla \varphi$ is a conformal Killing vector field in $(M,g)$.
\end{Lemma}

This immediately implies the following result.

\begin{Lemma} \label{lemma_lcw_ck}
Let $(M,g)$ be a simply connected open Riemannian manifold. The map
\[
\varphi \mapsto \abs{\nabla \varphi}^{-2} \nabla \varphi
\]
gives a bijective correspondence between the set of LCWs in $(M,g)$ modulo additive constants and the set of nonvanishing conformal Killing vector fields $X$ in $(M,g)$ satisfying $d(\abs{X}^{-2} X^{\flat}) = 0$. For any such $X$, there is a unique LCW $\varphi$ up to an additive constant that satisfies $\abs{X}^{-2} X = \nabla \varphi$.
\end{Lemma}
\begin{proof}
If $\varphi$ is an LCW, then $X = \abs{\nabla \varphi}^{-2} \nabla \varphi$ is a nonvanishing conformal Killing vector field by Lemma \ref{lemma_lcw_dksau} and $d(\abs{X}^{-2} X^{\flat}) = 0$. Conversely, if $X$ is as stated, then $\abs{X}^{-2} X = \nabla \varphi$ for some $\varphi \in C^{\infty}(M)$. Then $\abs{\nabla \varphi}^{-2} \nabla \varphi = X$ is a conformal Killing vector field, and $\varphi$ is an LCW by the previous lemma. The correspondence is bijective, since if $\abs{X}^{-2} X = \nabla \varphi_1 = \nabla \varphi_2$, then $\varphi_1$ and $\varphi_2$ differ by a constant.
\end{proof}

It follows that the classification of LCWs reduces to the determination of all nonvanishing conformal Killing vector fields satisfying $d(\abs{X}^{-2} X^{\flat}) = 0$. The classification of conformal Killing vector fields in $\mR^n$ is well known, but we give a proof for completeness.

\begin{Lemma} \label{lemma_ck_euclidean}
Let $\Omega \subset \mR^n$, $n \geq 3$, be open and connected. Then $X \in C^{\infty}(\Omega, \mR^n)$ is a conformal Killing vector field in $(\Omega, e)$ if and only if
\[
X(x) = (\alpha \cdot x)x - \frac{1}{2} \alpha \abs{x}^2 + cx + Bx + \gamma
\]
for some $\alpha, \gamma \in \mR^n$, $c \in \mR$, and some skew-symmetric matrix $B$.
\end{Lemma}
\begin{proof}
The equations for a conformal Killing vector field are
\[
\partial_j X_k + \partial_k X_j = \lambda \delta_{jk} \ \ \text{ in $\Omega$}, \qquad 1 \leq j, k \leq n,
\]
where $\lambda \in C^{\infty}(\Omega)$ is a scalar function (actually $\lambda = \frac{2}{n} \mathrm{div}(X)$). We begin by showing that $\lambda$ has zero Hessian. Taking derivatives in the previous equation gives
\begin{align*}
(\partial_a \lambda) \delta_{bc} &= \partial_{ab} X_c + \partial_{ac} X_b, \\
(\partial_b \lambda) \delta_{ac} &= \partial_{ab} X_c + \partial_{bc} X_a, \\
(\partial_c \lambda) \delta_{ab} &= \partial_{ac} X_b + \partial_{bc} X_a.
\end{align*}
These three equations imply that
\begin{equation} \label{partialab_xc}
2 \partial_{ab} X_c = (\partial_a \lambda) \delta_{bc} + (\partial_b \lambda) \delta_{ac} - (\partial_c \lambda) \delta_{ab}.
\end{equation}
Taking traces gives
\[
2 \Delta X_c = (2-n) \partial_c \lambda.
\]
Differentiating once gives
\begin{align*}
(2-n) \partial_{cd} \lambda &= 2 \Delta \partial_d X_c, \\
(2-n) \partial_{cd} \lambda &= 2 \Delta \partial_c X_d
\end{align*}
and thus
\[
(2-n) \partial_{cd} \lambda = \Delta (\partial_c X_d + \partial_d X_c) = (\Delta \lambda) \delta_{cd}.
\]
Taking traces gives that
\[
(2-n) \Delta \lambda = n \Delta \lambda
\]
which gives $\Delta \lambda = 0$. The previous equation gives $\partial_{cd} \lambda = 0$ for all $c, d$.

Since $\Omega$ is connected, we have $\lambda(x) = \alpha \cdot x + \beta$ for some $\alpha \in \mR^n$ and $\beta \in \mR$. Now \eqref{partialab_xc} implies
\[
2 \partial_{ab} X_c = \alpha_a \delta_{bc} + \alpha_b \delta_{ac} - \alpha_c \delta_{ab}.
\]
Integrating gives that
\[
X_j = \frac{1}{2}(\alpha_k \delta_{lj} + \alpha_l \delta_{kj} - \alpha_j \delta_{kl}) x_k x_l + m_{jk} x_k + \gamma_j
\]
for some $m_{jk}, \gamma_j \in \mR$. We finally compute
\[
\partial_j X_k + \partial_k X_j = 2 (\alpha \cdot x) \delta_{jk} + m_{jk} + m_{kj}
\]
and this is of the form $\lambda(x) \delta_{jk}$ if and only if $m_{jk} = c \delta_{jk} + b_{jk}$ for some $c \in \mR$ and some skew-symmetric matrix $B = (b_{jk})$.
\end{proof}

By Lemmas \ref{lemma_lcw_ck} and \ref{lemma_ck_euclidean}, the classification of all LCWs in $\mR^n$ follows once one checks which of the vector fields in Lemma \ref{lemma_ck_euclidean} satisfy $d(\abs{X}^{-2} X^{\flat}) = 0$. Presumably one could use Lemma \ref{lemma_lcw_ck} to characterize LCWs in other manifolds as well, provided that one can determine all conformal Killing vector fields.

\section{Action of the conformal group} \label{sec3}

The next lemma describes how the conformal group acts on LCWs and conformal Killing vector fields (in part (e) we also include multiplication by scalars).

\begin{Lemma} \label{lemma_conformal_group_action}
Let $\varphi$ be an LCW in a connected open set $\Omega \subset \mR^n$, $n \geq 3$, and let $X = \frac{d\varphi}{|d\varphi|_e^2}$ be the corresponding conformal Killing $1$-form written as
\[
X(x) = (\alpha \cdot x)x - \frac{1}{2} \alpha \abs{x}^2 + cx + Bx + \gamma.
\]
Let $F: \tilde{\Omega} \to \Omega$ be a conformal transformation from an open set $\tilde{\Omega} \subset \mR^{n}$ onto $\Omega$, i.e.\ $\kappa F^* e = e$ for some smooth positive $\kappa$. Then $\tilde{\varphi} = F^* \varphi$ is an LCW in $\tilde{\Omega}$ corresponding to conformal Killing $1$-form $\widetilde{X} = \kappa F^* X = \frac{d\tilde{\varphi}}{|d\tilde{\varphi}|_e^2}$, which be written as
\[
\widetilde{X}(x) = (\tilde{\alpha} \cdot x)x - \frac{1}{2} \tilde{\alpha} \abs{x}^2 + \tilde{c}x + \tilde{B}x + \tilde{\gamma}.
\]
\begin{enumerate}
\item[(a)]
If $F(x) = x-x_0$, then
\[
(\tilde{\alpha}, \tilde{c}, \tilde{B}, \tilde{\gamma}) = (\alpha, c-\alpha \cdot x_0, B + \alpha \wedge x_0, \gamma + (\alpha \cdot x_0) x_0 - \frac{1}{2} \abs{x_0}^2 \alpha - c x_0 - B x_0)
\]
where $\alpha \wedge \beta$ is the skew-symmetric matrix $(\alpha_j \beta_k - \beta_j \alpha_k)_{j,k=1}^n$.
\item[(b)]
If $F(x) = x/r$ for some $r \in \mR \setminus \{0\}$, then
\[
(\tilde{\alpha}, \tilde{c}, \tilde{B}, \tilde{\gamma}) = (\frac{1}{r}\alpha, c, B, r \gamma).
\]
\item[(c)]
If $F(x) = Rx$ for some matrix $R$ with $R^t R = I$, then
\[
(\tilde{\alpha}, \tilde{c}, \tilde{B}, \tilde{\gamma}) = (R^t \alpha, c, R^t B R, R^t \gamma).
\]
\item[(d)]
If $F(x) = \frac{x}{\abs{x}^2}$, then
\[
(\tilde{\alpha}, \tilde{c}, \tilde{B}, \tilde{\gamma}) = (-2\gamma, -c, B, -\frac{1}{2} \alpha).
\]
\end{enumerate}
Moreover:
\begin{enumerate}
\item[(e)]
If $k \in \mR \setminus \{0\}$, then $\tilde{\varphi} = k^{-1} \varphi$ is an LCW corresponding to $\widetilde{X}$, where
\[
(\tilde{\alpha}, \tilde{c}, \tilde{B}, \tilde{\gamma}) = (k \alpha, k c, k B, k \gamma).
\]\end{enumerate}
\end{Lemma}
\begin{proof}
The formulas in (a)--(d) hold for any conformal Killing $1$-form (not necessarily arising from an LCW). First note that if $\omega$ is a conformal Killing $1$-form in $(M,g)$ and if $F: (\tilde{M}, \tilde{g}) \to (M,g)$ satisfies $\kappa F^* g = \tilde{g}$ for some smooth positive $\kappa$, then $\kappa F^* \omega$ is a conformal Killing $1$-form in $(\tilde{M}, \tilde{g})$. To see this, write $\mu = F^{-*} \kappa$ and compute the total covariant derivative
\[
\nabla^{\tilde{g}} (\kappa F^* \omega) = \nabla^{F^*(\mu g)} F^*(\mu \omega) = F^* (\nabla^{\mu g} (\mu \omega)).
\]
Now the general identities (valid for any $1$-form $\omega$),
\begin{align*}
\nabla^{e^{2\lambda} g} \omega &= \nabla^g \omega - d\lambda \otimes \omega - \omega \otimes d\lambda + \langle d\lambda, \omega \rangle g, \\
\nabla^g(\mu \omega) &= \mu \nabla^g \omega + d\mu \otimes \omega
\end{align*}
and the properties of $\omega$ and $F$ imply that the symmetric part of $\nabla^{\tilde{g}}(\kappa F^* \omega)$ is a multiple of $\tilde{g}$, so $\kappa F^* \omega$ is indeed a conformal Killing $1$-form.

Now if $X$ is conformal Killing in a subset of $\mR^n$ and $\widetilde{X}$ is as above, then by Lemma \ref{lemma_ck_euclidean} $\widetilde{X}$ has the required representation in terms of $\tilde{\alpha}, \tilde{c}, \tilde{B}, \tilde{\gamma}$. If we identify $X$ and $\widetilde{X}$ with the corresponding vector fields in Cartesian coordinates, we have
\[
\widetilde{X} = \kappa (DF)^t (X \circ F), \qquad \kappa (DF)^t DF = I, \qquad \kappa = |\det DF|^{-2/n}
\]
where $DF = (\partial_k F_j)_{j,k=1}^n$. The claims (a)--(d) are a consequence of the following facts and short computations: \\

\noindent (a) If $F(x) = x-x_0$, then $\widetilde{X}(x) = X(x-x_0)$. \\[8pt]
(b) If $F(x) = x/r$, then $\widetilde{X}(x) = r X(x/r)$. \\[8pt]
(c) If $F(x) = Rx$, then $\widetilde{X}(x) = R^t X(R x)$. \\[8pt]
(d) If $F(x) = \frac{x}{\abs{x}^2}$, then $DF(x) = \frac{1}{\abs{x}^2}(I - 2 \hat{x} \otimes \hat{x})$ where $\hat{x} = \frac{x}{\abs{x}}$, so $\kappa = \abs{x}^4$ and
\begin{align*}
\widetilde{X}(x) &= \abs{x}^2 (I - 2 \hat{x} \otimes \hat{x}) X(\frac{x}{\abs{x}^2}) \\
 &= (I - 2 \hat{x} \otimes \hat{x}) \left[ \frac{1}{\abs{x}^2} ( (\alpha \cdot x) x - \frac{1}{2} \abs{x}^2 \alpha) + cx + Bx + \abs{x}^2 \gamma \right] \\
 &= \frac{1}{\abs{x}^2} ( (\alpha \cdot x) x - \frac{1}{2} \abs{x}^2 \alpha) + cx + Bx + \abs{x}^2 \gamma \\
 &\qquad - 2 \left[ \frac{1}{\abs{x}^2} ( (\alpha \cdot x) x - \frac{1}{2} (\alpha \cdot x) x) + cx + (Bx \cdot \hat{x}) \hat{x} + (\gamma \cdot x) x \right].
\end{align*}
Here $Bx \cdot \hat{x} = 0$ because of skew-symmetry. Thus
\[
\widetilde{X}(x) = -\frac{1}{2} \alpha - cx + Bx + \abs{x}^2 \gamma - 2(\gamma \cdot x) x.
\]

Finally, if $\varphi$ is an LCW corresponding to $X = \frac{d\varphi}{|d\varphi|_e^2}$ and $F$ if is conformal, it follows that $\tilde{\varphi} = F^* \varphi$ is an LCW corresponding to $\widetilde{X} = \frac{d\tilde{\varphi}}{|d\tilde{\varphi}|_e^2} = \kappa F^* X$. Part (e) is trivial.
\end{proof}

\section{Classification of LCWs in the Euclidean space} \label{sec4}

We will now give the new proof of the classification of LCWs in $\mR^n$, $n \geq 3$ (Theorem \ref{thm_lcw_rn_classification}), and the orbits of the action by precomposition with a conformal transformation and post composition with a linear function (Theorem \ref{thm_lcw_rn_orbits}).


Before the proofs, we give two lemmas. The first lemma is a simple consequence of the equation $d(|X|^{-2} X) = 0$ that holds for conformal Killing vector fields $X$ associated to LCWs.

\begin{Lemma} \label{lemma_dx_wedge_x}
If a nonvanishing $1$-form $X$ satisfies $d(|X|^{-2} X) = 0$, then
\begin{align*}
dX \wedge X = 0.
\end{align*}
\end{Lemma}
\begin{proof}
We compute
\[
d(|X|^{-2} X) = |X|^{-2} dX - |X|^{-4} d(|X|^2) \wedge X.
\]
Thus $d(|X|^{-2} X) = 0$ if and only if
\[
dX = |X|^{-2} d(|X|^2) \wedge X.
\]
This means that the two-form $dX$ is proportional to a wedge product of $1$-forms. In particular, one obtains the three equations
\begin{align*}
dX \wedge X &= 0, \\
dX \wedge d(|X|^2) &= 0, \\
|d(|X|^2) \wedge X| &= |X|^2 |dX|.
\end{align*}
(Conversely, one can show that in any connected manifold these three equations imply that $dX = \pm |X|^{-2} d(|X|^2) \wedge X$.)
\end{proof}

Next we use the equation in Lemma \ref{lemma_dx_wedge_x} to derive some restrictions to the parameters of $X = (\alpha, c, B, \gamma)$ arising from an LCW.

\begin{Lemma} \label{lemma_lcw_parameter_conditions}
Let $\varphi$ be an LCW in $\Omega \subset \mR^n$, $n \geq 3$. The corresponding conformal Killing vector field $X = (\alpha, c, B, \gamma)$ satisfies
\begin{align*}
B \wedge \gamma &= 0, \\
cB &= \alpha \wedge \gamma
\end{align*}
where $B = b_{kj} \,dx_j \wedge dx_k$.
\end{Lemma}
\begin{proof}
We write $X = X_k \,dx_k$, where
\[
X_k = (\alpha_j x_j) x_k - \frac{1}{2} \alpha_k |x|^2 + c x_k + b_{kj} x_j + \gamma_k.
\]
It follows that
\begin{align*}
dX &= dX_k \wedge dx_k \\
 &= (\alpha_j x_k - \alpha_k x_j + b_{kj}) \,dx_j \wedge dx_k \\
 &= \alpha \wedge p + B
\end{align*}
where $\alpha = \alpha_j \,dx_j$, $p = x_j \,dx_j$ and $B = b_{kj} \,dx_j \wedge dx_k$. Similarly, in this notation,
\[
X = \langle \alpha, p \rangle p - \frac{1}{2} |p|^2 \alpha + c p + i_{p^{\sharp}} B + \gamma.
\]
Thus Lemma \ref{lemma_dx_wedge_x} gives that
\begin{align*}
0 &= dX \wedge X \\
 &= \alpha \wedge p \wedge (i_{p^{\sharp}} B + \gamma) + B \wedge (\langle \alpha, p \rangle p - \frac{1}{2} |p|^2 \alpha + c p + i_{p^{\sharp}} B + \gamma).
\end{align*}
The components of this $3$-form are second order polynomials in $x$ and vanish for all $x \in \Omega$, hence also for all $x \in \mR^n$. Evaluating at $x=0$ yields the equation
\[
B \wedge \gamma = 0.
\]
Moreover, looking at first order terms in $x$ gives that
\[
\alpha \wedge p \wedge \gamma + B \wedge (cp + i_{p^{\sharp}} B) = 0.
\]
By the properties of the interior product we have $B \wedge i_{p^{\sharp}} B = 0$. In fact,
\[
0 = i_{p^{\sharp}}(B \wedge B) = (i_{p^{\sharp}} B) \wedge B + B \wedge i_{p^{\sharp}} B = 2 B \wedge i_{p^{\sharp}} B.
\]
Hence $(cB - \alpha \wedge \gamma) \wedge p = 0$, and evaluating at $x = e_j$ for $1 \leq j \leq n$ we obtain
\[
cB = \alpha \wedge \gamma. \qedhere
\]
\end{proof}

\begin{proof}[Proof of Theorem \ref{thm_lcw_rn_classification}]
Let $\varphi$ be an LCW in a connected open set $\Omega \subset \mR^n$, $n \geq 3$, and let $X = (\alpha, c, B, \gamma)$ be the corresponding conformal Killing vector field. We divide the argument in several cases. In each case, we will use Lemma \ref{lemma_conformal_group_action} (in fact only the translation property in part (a)) to reduce conformal Killing vector fields to simpler ones that can be realized by explicit LCWs. We omit the routine computations required to show that $d(|\nabla \varphi|^{-2} \nabla \varphi) = 0$ for each of the functions $\varphi$ given below (so that indeed these functions are LCWs). \\

\noindent {\it Case 1.} $\alpha = 0$. \\

Since $\alpha = 0$, from Lemma \ref{lemma_lcw_parameter_conditions} we obtain that $cB = 0$. This gives two subcases. \\

\noindent {\it Case 1a.} $\alpha = 0$ and $B = 0$. \\

Then $X = (0, c, 0, \gamma)$. If $c = 0$, then $\gamma \neq 0$ (since $X$ is nonvanishing) and $X = \gamma$ is realized by the LCW $\varphi(x) = \abs{\gamma}^{-2} \gamma \cdot x$. On the other hand, if $c \neq 0$, after a translation $F(x) = x - \gamma/c$ we may bring $X$ to the form $(0,c,0,0)$, and this vector field is realized by the LCW $\varphi(x) = c^{-1} \log |x|$. \\

\noindent {\it Case 1b.} $\alpha = 0$ and $c = 0$. \\

In this case $X = (0, 0, B, \gamma)$. First suppose that $\gamma = 0$. Since then $B \neq 0$, we can choose a vector $x_0$ with $Bx_0 \neq 0$ and apply the translation $F(x) = x + x_0$ to bring $X$ to the form $(0, 0, B, Bx_0)$. We may thus assume that $\gamma \neq 0$. If $B = 0$ we are again in Case 1a, so we may assume that $B \neq 0$. From Lemma \ref{lemma_lcw_parameter_conditions} we have $B \wedge \gamma = 0$, which implies that $B = \gamma \wedge \sigma$ for some vector $\sigma \neq 0$, and we may further assume that $\sigma$ is orthogonal to $\gamma$. After a translation $F(x) = x - \sigma/|\sigma|^2$, $X$ becomes $(0, 0, a \hat{\gamma} \wedge \hat{\sigma}, 0)$ where $a = \abs{\gamma} \abs{\sigma} > 0$. This is realized by the LCW $\varphi(x) = a^{-1} \mathrm{arctan} \frac{\hat{\gamma} \cdot x}{\hat{\sigma} \cdot x}$. \\

\noindent {\it Case 2.} $\alpha \neq 0$. \\

In this case $X = (\alpha, c, B, \gamma)$. If $c \neq 0$, we first use the translation $F(x) = x - c |\alpha|^{-2} \alpha$ to reduce to the case $X = (\alpha, 0, B, \gamma)$ for some new $\gamma$. If $\gamma = 0$, then $X = (\alpha, 0, B, 0)$, and if $x_0$ is a vector with $\alpha \cdot x_0 = 0$ we apply the translation $F(x) = x + x_0$ to bring $X$ to the form $(\alpha, 0, \tilde{B}, B x_0 - \frac{1}{2} |x_0|^2 \alpha )$ for some skew-symmetric $\tilde{B}$. If $|x_0|$ is large enough one has $B x_0 - \frac{1}{2} |x_0|^2 \alpha \neq 0$. This means that we may assume that $X$ is of the form $(\alpha, 0, B, \gamma)$ for some $\gamma \neq 0$.

Now, let $X = (\alpha, 0, B, \gamma)$ with $\gamma \neq 0$. By Lemma \ref{lemma_lcw_parameter_conditions} we have $cB = \alpha \wedge \gamma$, and since $c = 0$ this implies that $\gamma = r\alpha$ for some $r \in \mR \setminus \{0\}$. Thus we may assume that $X = (\alpha, 0, B, r\alpha)$ for some $r \neq 0$. By Lemma \ref{lemma_lcw_parameter_conditions} again we have $B \wedge r\alpha = 0$, which implies that $B = \alpha \wedge \sigma$ for some vector $\sigma$ where we may assume that $\alpha \cdot \sigma = 0$. Thus we have $X = (\alpha, 0, \alpha \wedge \sigma, r\alpha)$. After a translation $F(x) = x + \sigma$, the vector field reduces to $X = (\alpha, 0, 0, s \alpha)$ where $s = r+\frac{1}{2} |\sigma|^2$.

It remains to find LCWs that realize $X = (\alpha, 0, 0, s\alpha)$ for different values of $s$:
\begin{itemize}
\item If $s > 0$, the LCW $\varphi(x) = \frac{1}{\sqrt{s} |\alpha|} \,\mathrm{arctan}\,\frac{-2\hat{\alpha} \cdot x/\sqrt{s}}{|x|^2/s - 1}$ corresponds to $X = (\alpha, 0, 0, \frac{s}{2}\alpha)$.
\item If $s = 0$, the LCW $\varphi(x) = -\frac{2}{|\alpha|} \hat{\alpha} \cdot \frac{x}{|x|^2}$ corresponds to $X = (\alpha, 0, 0, 0)$.
\item If $s > 0$, the LCW $\varphi(x) = \frac{1}{\sqrt{s} |\alpha|} \,\mathrm{arctanh}\,\frac{-2\hat{\alpha} \cdot x/\sqrt{s}}{|x|^2/s + 1}$ yields the vector field $X = (\alpha, 0, 0, -\frac{s}{2} \alpha)$.
\end{itemize}
This concludes the proof.
\end{proof}

\begin{proof}[Proof of Theorem \ref{thm_lcw_rn_orbits}]
We have already shown that precomposition with a translation $F(x)=x-x_0$ and post-composition with a linear function reduces any LCW to one of the six families of LCWs in Theorem \ref{thm_lcw_rn_classification}.
After composing with a rotation $R$, we can take $\gamma=e_1$ and $\sigma=e_2$.
The scalar parameter $s>0$ can be reduced to $1$ by a combination of composition with a dilation of scale $r=\sqrt{s}$ and post-multiplication by a constant $k=\frac{1}{\sqrt{s}}$.
Therefore, any LCW on $\mR^n$ is reduced to one of the following six representative LCWs:

\[
\begin{array}{lll}
e_1 \cdot x, &\quad
\log |x|, &\quad
\mathrm{arctan} \,\frac{e_1 \cdot x}{e_2 \cdot x}, \\[7pt]
- 2 \frac{e_1 \cdot x}{|x|^2}, &\quad
\mathrm{arctanh}\,\frac{-2e_1 \cdot x}{|x|^2 + 1}, &\quad
\mathrm{arctan}\,\frac{-2 e_1 \cdot x}{|x|^2 - 1}.
\end{array}
\]

Observe that up to this moment, to arrive to these six representatives, we have only used affine conformal transformations.
Using also the inversion, we will show that the LCWs in the same column in the above listing belong to the same orbit.

Composition of $e_1\cdot x$ with the inversion $F(x)=x/|x|^2$ gives $\frac{e_1 \cdot x}{|x|^2}$.
The relation between $\log|x|$ and $\mathrm{arctanh}\,\frac{-2e_1 \cdot x}{|x|^2 + 1}$ is seen more easily looking at the corresponding conformal Killing vector fields:
\begin{itemize}
  \item The translation $F(x) = x-e_1$ acts on $(0,1,0,0)$ (which corresponds to $\log|x|$) and gives $(0,1,0,-e_1)$.
  \item The inversion $F(x)=x/|x|^2$ acts on $(0,1,0,-e_1)$ and gives $(2e_1,-1,0,0)$.
  \item The translation $F(x) = x+\frac{1}{2} e_1$ acts on  $(2e_1,-1,0,0)$ and gives  $(2e_1,0,0,-\frac{1}{4}e_1)$.
  \item The dilation $F(x) = x/2$ acts on $(2e_1,0,0,-\frac{1}{4}e_1)$ and gives $(e_1,0,0,-\frac{1}{2}e_1)$, which corresponds to $\mathrm{arctanh}\,\frac{-2e_1 \cdot x}{|x|^2 + 1}$.
\end{itemize}
A similar argument works for the $\mathrm{arctan}$ weights.

It remains to prove that the three remaining orbits are different.
Let us assume that $\varphi=e_1\cdot x$ is $b+a(\psi\circ F)$, for $\psi=\log|x|$ and $F$ a conformal transformation in $\Omega\subset\mR^n$.
According to Liouville's theorem (see e.g \cite[Theorem 2.3.1]{IwaniecMartin}), $F$ is the restriction to $\Omega$ of:
\[
  F(x) = x_1+{\frac  {\alpha A(x-x_0)}{|x-x_0|^{\epsilon }}}
\]
for $x_0\in\mR^n\setminus\Omega,x_1\in\mR^n$, an orthogonal matrix $A$, $\alpha\in\mR\setminus\{ 0\}$, $\epsilon\in\{0,2\}$.

This is a composition of elementary transformations whose effect we know thanks to Lemma \ref{lemma_conformal_group_action}, so we can apply them to the conformal Killing vector field $X=(0,1,0,0)$ that corresponds to $\psi$.
We will show that $F^\ast X$ is never equal to the conformal Killing vector field of $\varphi$, which is $X=(0,0,0,e_1)$.

Let us consider the case of $\epsilon=0$ first:
\begin{itemize}
  \item The translation $F(x) = x-x_0$ acting on $(0,1,0,0)$ gives $(0,1,0,-x_0)$.
  \item The rotation $F(x)=Rx$ acting on $(0,1,0,-x_0)$ gives $(0,1,0,-R^tx_0)$.
  \item The dilation $F(x)=x/r$ acting on $(0,1,0,-R^tx_0)$ gives $(0,1,0,-rR^tx_0)$.
  \item Finally, the translation $F(x) = x-x_1$ acting on $(0,1,0,-rR^tx_0)$ must give a multiple of $(0,0,0,e_1)$, but the constant $\widetilde{c}=c=1$.
\end{itemize}

If $\epsilon=2$:
\begin{itemize}
  \item The translation $F(x) = x-x_0$ acting on $(0,1,0,0)$ gives $(0,1,0,-x_0)$.
  \item The inversion $F(x)=x/|x|^2$ acting on $(0,1,0,-x_0)$ gives $(2x_0,-1,0,0)$.
  \item The rotation $F(x)=Rx$ acting on $(2x_0,-1,0,0)$ gives $(2R^tx_0,-1,0,0)$.
  \item The dilation $F(x)=x/r$ acting on $(2R^tx_0,-1,0,0)$ gives $((2/r)R^tx_0,-1,0,0)$.
  \item Finally, the translation $F(x) = x-x_1$ acting on $((2/r)R^tx_0,-1,0,0)$ must give a multiple of $(0,0,0,e_1)$, which implies that $(2/r)R^tx_0=0$, and $\widetilde{c}$ cannot be zero.
\end{itemize}

We have proven that the orbits of $e_1\cdot x$ and $\log|x|$ by the conformal group are different.
Analogous computations show that $\mathrm{arctan} \,\frac{e_1 \cdot x}{e_2 \cdot x}$ also belongs to a different orbit.
Postcomposition with a linear function simply scales the conformal Killing vector field, and this action commutes with precomposition with a conformal transformation, so this action cannot mix the orbits.
\end{proof}

\begin{Remark}
The fact that there are only three orbits for the LCWs under the action of the conformal group can also be seen directly as follows. Clearly 
\[
\frac{e_1 \cdot x}{|x|^2} = F^*(e_1 \cdot x), \qquad F(x) = \frac{x}{|x|^2}.
\]
Next define the conformal map 
\[
F(x) = \frac{x-e_1}{|x-e_1|^2} + \frac{1}{2} e_1.
\]
It is easy to see that 
\[
F^*(\mathrm{arctan} \,\frac{e_2 \cdot x}{e_1 \cdot x})(x) = \mathrm{arctan} \,\frac{2e_2 \cdot x}{|x|^2-1}.
\] 
Moreover, since $|F(x)|^2 = \frac{|x+e_1|^2}{4 |x-e_1|^2}$, one has 
\[
F^*(\log\,2|x|) = \frac{1}{2} \log\,\frac{|x+e_1|^2}{|x-e_1|^2} = \frac{1}{2} \log\,\frac{1+\frac{2x_1}{|x|^2+1}}{1-\frac{2x_1}{|x|^2+1}} = \mathrm{arctanh} \,\frac{2x_1}{|x|^2+1}.
\]
\end{Remark}

\section{On orthogonal LCWs} \label{sec5}

In this section, we show how the existence of several orthogonal LCWs results in a special conformal product structure.

\begin{defn}
We say that the LCWs $\varphi_1, \ldots, \varphi_m$ for $g$ are \emph{orthogonal} if their gradients $\mathrm{grad}_g(\varphi_1), \ldots, \mathrm{grad}_g(\varphi_m)$ are orthogonal with respect to $g$.
\end{defn}

\begin{thm}
Let $(M,g)$ be an $n$-dimensional  Riemannian manifold. Let $p\in M$, and suppose that in an open neighbourhood of $p$, the conformal class $[g]$ admits $m$ orthogonal LCWs  $\varphi_1,\dots,\varphi_m$. Then there is a set of coordinates about $p$, $\Phi=(z_1,\dots, z_n)$, such that
\begin{enumerate}
\item $z_1=\varphi_1,\dots, z_m=\varphi_m$;
\item some conformal multiple of $g$ has the local expression
\[
\left[\begin{array}{ccccc}
1 & 0 & \dots & 0  & 0\\
0 & f_2(z_{m+1},\dots,z_n) & \dots & 0 & 0\\
 &  & \vdots &  & \\
0 & 0 & \dots & f_m(z_{m+1},\dots,z_n) & 0\\
0 & 0 & \dots & 0 & D(z_{m+1},\dots,z_n)\\
\end{array}\right]
\]
where $D$ is an $(n-m)\times (n-m)$ symmetric matrix.
\end{enumerate}
Conversely, if in some local coordinates $(z_1, \ldots, z_n)$ the metric has the above form up to a conformal multiple, then $z_1, \ldots, z_m$ are orthogonal LCWs.
\end{thm}

\begin{proof}
Let $\varphi_1, \ldots, \varphi_m$ be orthogonal LCWs. Using \cite[Theorem 1.2 and its proof]{DKSaU}, for each $j=1,\dots,m$ there is 
a smooth conformal factor $f_j$ near $p$ such that the gradient
of $\varphi_j$ in the metric $f_j g$
is a \emph{unit parallel vector field} with respect to $f_j g$. In other words, writing 
\[
X_j = \mathrm{grad}_{f_j g}(\varphi_j),
\]
we have $\nabla^{f_j g}{X_j}\equiv 0$ and $f_j g(X_j, X_j) = 1$. After passing to a conformal multiple of $g$, we can assume without loss of generality that $f_1 \equiv 1$.

The gradient of $\varphi_j$ in the metric $g$ is given by $f_j\cdot X_j$, since
\[
\mathrm{grad}_{g}(\varphi_j) = f_j \,\mathrm{grad}_{f_j g}(\varphi_j) = f_j X_j.
\]

Using the formula for the Levi-Civita connection of a conformal metric, we have 
\begin{equation}\label{eqn: connection for a multiple of the metric}
\nabla_X^{f_j g} Y = \nabla_X^g Y + (Xh_j) Y + (Yh_j) X - g(X, Y) \mathrm{grad}_g(h_j)
\end{equation}
where $h_j = \frac{1}{2} \log\,f_j$. Thus in particular 
\[
    0=\nabla^{f_j g}_{X_k}X_j=\nabla^g_{X_k}X_j+X_k (h_j)\cdot X_j+X_j(h_j)\cdot X_k-g(X_j,X_k)\grad_g(h_j).
\]
Since $g(X_j,X_k)=0$ for $j \neq k$, $\nabla^g_{X_k}X_j\in\spann\{X_j,X_k\}$.
Interchanging $j$ and $k$, we get also that
$\nabla^g_{X_j}X_k\in\spann\{X_j,X_k\}$, and
 this implies that
\[
[X_j,X_k]=
\nabla^g_{X_j}X_k-\nabla^g_{X_k}X_j
\in\spann\{X_j,X_k\}.
\]
Thus the distribution generated by $X_j$ and $X_k$ is integrable, and
it also follows that the whole distribution $\spann\{X_1,\ldots,X_m\}$ is integrable.

Using the Frobenius theorem, there is a chart $(y_1, \ldots, y_n)$ near $p$ such that the sets $\{ (y_{m+1}, \ldots, y_n) = \mathrm{const} \}$ are integral manifolds of the distribution $\spann\{X_1,\ldots,X_m\}$. We define another chart $(z_1, \ldots, z_n)$ near $p$ such that $z_j = \varphi_j$ for $1 \leq j \leq m$ and $z_l = y_l$ for $l \geq m+1$.
Since $X_j = \mathrm{grad}_{f_j g}(\varphi_j)$ is tangent to the manifolds $\{ (y_{m+1}, \ldots, y_n) = \mathrm{const} \}$, and since the $X_j$ are orthogonal, it follows that the differentials of the functions $z_1, \ldots, z_n$ are linearly independent, and $(z_1, \ldots, z_n)$ is a chart.
We claim that for this chart,
\begin{equation} \label{xk_dzk}
X_k = \partial_{z_k}, \qquad 1 \leq k \leq m.
\end{equation}
To prove \eqref{xk_dzk}, we note that for $1 \leq j, k \leq m$ one has 
\[
    X_j(\varphi_k) = d \varphi_k(X_j) = g(X_j, \mathrm{grad}_g(\varphi_k)) = f_k g(X_j, X_k) = \delta_{jk}.
\]
Here we used that $\mathrm{grad}_g(\varphi_k) = f_k X_k$, the $g$-orthogonality of the $X_j$, and the fact that $f_k g(X_k, X_k) = 1$. Moreover, $X_k(z_l) = dz_l(X_k) = 0$ whenever $l \geq m+1$ and $k \leq m$, since the manifolds $\{(z_{m+1}, \ldots, z_n)=\mathrm{const} \}$ are integral manifolds of $\spann\{X_1,\ldots,X_m\}$. Thus $X_k = \sum_{l=1}^n X_k(z_l) \partial_{z_l} = \partial_{z_k}$ for $1 \leq k \leq m$, which proves \eqref{xk_dzk}.

In particular, by \eqref{xk_dzk} the Lie brackets $[X_j, X_k]$ are zero.
Next, we use that the Levi-Civita connection is torsion free for any metric, that for any $j$, the vector field $X_j$ is parallel for the metric $f_j g$, and finally the  formula \eqref{eqn: connection for a multiple of the metric} for the conformal factor $\frac{f_j}{f_k}$ with $h_{jk} = \frac{1}{2} \log \frac{f_j}{f_k}$. These facts imply that for $j \neq k$ one has 
\begin{align*}
    0&=[X_j,X_k]=\nabla^{f_j g}_{X_j}X_k - \nabla^{f_j g}_{X_k}X_j
    =\nabla^{f_j g}_{X_j}X_k = \nabla^{(f_j/f_k) f_k g}_{X_j}X_k\\
    &=\nabla^{f_k g}_{X_j}X_k + X_j(h_{jk})X_k + X_k(h_{jk})X_j - f_k g(X_j,X_k)\grad_{f_k g}(h_{jk})\\
    &=X_k(h_{jk})X_j + X_j(h_{jk})X_k.
\end{align*}
Since $X_j=\partial_j$ and $X_k=\partial_k$ are linearly independent, we deduce that $X_k(h_{jk}) = \partial_k (h_{jk}) = 0$ for $j \neq k$. Further, since $h_{jk} = h_j - h_k$ where $h_j = \frac{1}{2} \log\,f_j$, we obtain that 
\[
\partial_k(h_j - h_k) = 0, \qquad 1 \leq j, k \leq m.
\]
One also has $h_1 \equiv 0$, since $f_1 \equiv 1$.

Next, we claim that the functions $h_j$, $1 \leq j \leq m$, only depend on the variables $z_{m+1}, \ldots, z_n$. First note that 
\[
\partial_1 h_j = \partial_1 (h_j - h_1) = 0, \qquad 1 \leq j \leq m.
\]
Next we use that $\partial_2 h_2 = \partial_2 (h_2-h_1) = 0$, which implies that 
\[
\partial_2 h_j = \partial_2 (h_j - h_2) = 0, \qquad 1 \leq j \leq m.
\]
Repeating this argument shows that indeed $h_j$, $1 \leq j \leq m$, only depends on the variables $z_{m+1}, \ldots, z_n$.


Combining the above facts, the metric in the $(z_1, \ldots, z_n)$ coordinates has the following form:

\begin{equation}
\label{eq:metric}
\left[\begin{array}{ccccc}
1 & 0 & \dots & 0  & 0\\
0 & f_2(z_{m+1},\dots,z_n) & \dots & 0 & 0\\
 &  & \vdots &  & \\
0 & 0 & \dots & f_m(z_{m+1},\dots,z_n) & 0\\
0 & 0 & \dots & 0 & D(z_{1},\dots,z_n)\\
\end{array}\right]
\end{equation}
where $D$ is an $(n-m)\times (n-m)$ symmetric matrix.

For any element $D_{ab}$ of $D$, where $a, b \geq m+1$, and for any $j\leq m$ we can compute
\[
    \partial_j D_{ab} = \partial_j (g(\partial_a, \partial_b)) =
    g(\nabla^g_{\partial_j}\partial_a, \partial_b) + g(\nabla^g_{\partial_j}\partial_b, \partial_a)
\]
Using \eqref{eqn: connection for a multiple of the metric}, the first term in the above sum can be computed as
\[
\begin{split}
    g(\nabla^g_{\partial_j}\partial_a, \partial_b) &=
    g(\nabla^g_{\partial_a}\partial_j, \partial_b) \\
    &= g(\nabla^{f_j g}_{\partial_a}\partial_j - \partial_a(h_j)\partial_j -\partial_j(h_j)\partial_a +g(\partial_a,\partial_j)\grad_g(h_j), \partial_b )
        \\
    &=0
\end{split}
\]
since the metric is of the form in \eqref{eq:metric}. Similarly one has $g(\nabla^g_{\partial_j}\partial_b, \partial_a) = 0$.

We recall that we assumed $f_1 \equiv 1$, so the original metric in the $(z_1, \ldots, z_n)$ coordinates is a conformal multiple of
\[
\left[\begin{array}{ccccc}
1 & 0 & \dots & 0  & 0\\
0 & f_2(z_{m+1},\dots,z_n) & \dots & 0 & 0\\
 &  & \ddots &  & \\
0 & 0 & \dots & f_m(z_{m+1},\dots,z_n) & 0\\
0 & 0 & \dots & 0 & D(z_{m+1},\dots,z_n)\\
\end{array}\right]
\]
where $D$ is an $(n-m)\times (n-m)$ symmetric matrix.

Finally, to show the converse statement we assume that the metric is conformal to the above form in some  local coordinates $(z_1, \ldots, z_n)$. Using \cite[Remark 2.8]{DKSaU}, $z_1$ is an LCW. Similarly, after taking $f_2$ as a factor, the metric is conformal to 
\[
\left[\begin{array}{ccccc}
\tilde{f}_1(z_{m+1},\dots,z_n) & 0 & \dots & 0  & 0\\
0 &1& \dots & 0 & 0\\
 &  & \ddots &  & \\
0 & 0 & \dots & \tilde{f}_m(z_{m+1},\dots,z_n) & 0\\
0 & 0 & \dots & 0 & \tilde{D}(z_{m+1},\dots,z_n)\\
\end{array}\right]
\]
where $\tilde{f}_1, \tilde{f}_3, \ldots, \tilde{f}_m, \tilde{D}$ only depend on $z_{m+1}, \ldots, z_n$. Thus $z_2$ is an LCW again by \cite[Remark 2.8]{DKSaU}. Similarly we get that $z_1, z_2, \ldots, z_m$ are all LCWs, and they are clearly orthogonal.
\end{proof}

\begin{cor}
    An $n$-dimensional manifold with $n$ orthogonal LCWs is conformally flat.
\end{cor}

\section{Limiting Carleman weights in 3-manifolds} \label{sec6}

Conformal geometry in dimension three is usually studied through the use of the Cotton tensor of the metric.
This is a $(3,0)$-covariant tensor that is antisymmetric in its last two entries;
therefore it can be seen as a $2$-form valued tensor. In dimension 3, we can identify these with vectors, thus resulting in the so called Cotton-York tensor, that contains the same information as the original Cotton tensor.

 In \cite{AFGR} and \cite{AFG} we studied the restrictions that  a limiting Carleman weight sets on the Cotton-York tensor.

 \begin{thm}[\cite{AFGR}]\label{Cotton2}
  Let $n=3$. If a metric $\tilde{g} \in [g]$ admits a parallel vector field, then for any $p\in M$, there is a tangent vector $v\in T_pM$ such that
  $$
  CY_p(v,v)=CY_p(w_1,w_2)=0
  $$
  for any pair of vectors $w_1,w_2 \in v^\perp$. Moreover, if $\varphi$ is an LCW in $(M,g)$, then $v = \mathrm{grad}_{g}(\varphi)|_p$ has the above property.
 \end{thm}

In analogy with the higher dimensional case we still call such   a direction $v\in T_pM$ an \emph{eigenflag}.
Thus limiting Carleman weights line along eigenflag directions, but there is a priori no reason why eigenflag directions should be realized by limiting Carleman weights.

Observe also that the above theorem has the simple algebraic consequence that the determinant of $CY_p$ vanishes whenever $CY_p$ has the form of Theorem \ref{Cotton2} \cite[Corollary 1.7]{AFGR}. Thus
 there is a sequence of implications
\begin{quotation}
 Existence of LCW $\Rightarrow$
 existence of eigenflag
 $\Rightarrow
\det CY_p=0$.
\end{quotation}

The above implications raise the following natural questions:
\begin{itemize}
\item Are there metrics with eigenflag directions that do not correspond to limiting Carleman weights?
\item  Does $\det(CY_p)=0$ in an open set imply the existence of a limiting Carleman weight?
\item How many different limiting Carleman weights can a metric admit without being conformally flat?
\item More generally, can we characterize metrics that have more than one limiting Carleman weight?
\end{itemize}

The rest of this section answers the above questions.


\subsection{3-manifolds with no LCWs whose Cotton-York tensor has eigenflag directions}
Here we will prove Theorem \ref{thm:3-dim-unimodular group}.
We will provide such examples in unimodular Lie groups. A good reference for such manifolds is \cite{Mil}.

Let $G$ be a 3-dimensional Lie group with Lie algebra $\lieg$. Recall from \cite[Lemma 4.1]{Mil} that if $G$ is unimodular,
 there exists a self-adjoint map
 $L:\lieg\to\lieg$ such that for any $u,v\in\lieg$,
\[
[u,v]=L(u\times v).
\]
The vector product in $\lieg$ is the one defined by an oriented orthonormal basis $e_1, e_2, e_3$ consisting of eigenvectors of $L$, thus there are
$\lambda_1,\lambda_2,\lambda_3\in\R$ such that
\[
[e_i,e_j]=\lambda_k e_k
\]
where $i,j,k$ is an oriented set of indices.
%

Following Milnor, to compute the curvature tensors of $G$, we define numbers $\mu_1,\mu_2,\mu_3$ as
\[
\mu_i=\dfrac{\lambda_1+\lambda_2+\lambda_3}{2}-\lambda_i.
\]

From \cite[Theorem 4.3]{Mil}, the Ricci tensor in the basis $e_1,e_2,e_3$ diagonalizes, with
\[
\Ric(e_1)=2\mu_2\mu_3, \quad \Ric(e_2)=2\mu_1\mu_3, \quad \Ric(e_3)=2\mu_1\mu_2,
\]
and the scalar curvature is
\[
s=2(\mu_1\mu_2+\mu_1\mu_3+\mu_2\mu_3).
\]

In order to compute the Cotton tensor, we calculate first the Levi-Civita connection for the vector fields $e_1, e_2, e_3$.
\cite[equation (5.4)]{Mil} gives that
\[
\nabla_{e_i}e_j=\sum_k \frac{1}{2}\left(\alpha_{ijk} -\alpha_{jki} +\alpha_{kij}\right)e_k,
\]
with
\[
\alpha_{ijk}=\left\langle [e_i,e_j],e_k\right\rangle=\lambda_k.
\]
Writing this in terms of the $\lambda_i$'s yields
\[
\nabla_{e_1}e_2= \frac{\lambda_3-\lambda_1+\lambda_2}{2}e_3, \quad
\nabla_{e_2}e_3= \frac{\lambda_1-\lambda_2+\lambda_3}{2}e_1 , \quad
\nabla_{e_3}e_1=  \frac{\lambda_2-\lambda_3+\lambda_1}{2}e_2,
\]
\[
\nabla_{e_2}e_1= \frac{-\lambda_3-\lambda_1+\lambda_2}{2}e_3, \quad
\nabla_{e_3}e_2= \frac{-\lambda_1-\lambda_2+\lambda_3}{2}e_1 , \quad
\nabla_{e_1}e_3=  \frac{-\lambda_2-\lambda_3+\lambda_1}{2}e_2.
\]
In what follows we will denote by
\[
n_{ijk}=\left\langle \nabla_{e_i}e_j,e_k\right\rangle,
\]
and observe that  $n_{ijk}=0$ whenever one of the indices repeats.
Observe also that the $n_{ijk}$ are \emph{not antisymmetric} in the first two indices, since this would otherwise force the Lie algebra to be abelian.

The Schouten tensor in dimension 3 is
\[
S=\Ric-\frac{s}{4}g,
\]
and
the formula for the Cotton tensor (see \cite[Section 2]{AFGR}) is given as
\begin{align*}
C_{ijk} = C(e_i,e_j,e_k) =
\nabla_{e_i}S(e_j,e_k)-\nabla_{e_j}S(e_i,e_k).
\end{align*}
Since $\nabla_{e_i}S(e_j,e_k) = e_i(S(e_j,e_k)) - S(\nabla_{e_i} e_j, e_k) - S(e_j, \nabla_{e_i} e_k)$ and since $S(e_j, e_k)$ is constant in $x$, the above formulas imply that for any $i, j, k$ one has 
\[
C_{ijk} = -n_{ijk} S_{kk} - n_{ikj} S_{jj} + n_{jik} S_{kk} + n_{jki} S_{ii}.
\]
Thus $C_{ijk} = 0$ whenever two of the indices coincide. The only nonzero terms are
\begin{align*}
C_{123} &= -n_{123}S_{33}-n_{132}S_{22}+n_{213}S_{33}+n_{231}S_{11}\\
C_{231} &= -n_{231}S_{11}-n_{213}S_{33}+n_{321}S_{11}+n_{312}S_{22}\\
C_{312} &= -n_{312}S_{22}-n_{321}S_{11}+n_{132}S_{22}+n_{123}S_{33}.
\end{align*}
The Cotton-York tensor is obtained by dualizing the first two indices, thus
\[
CY_{11}=C_{231}, \quad CY_{22}=C_{312}, \quad CY_{33}=C_{123},
\]
while the rest of the components are zero.

As the reader can easily check by a tedious computation,
by choosing $\lambda_1=6$, $\lambda_2=-4$, and $\lambda_3=5$, one obtains
\[
CY_{11}=\frac{-315}{2}, \quad CY_{22}=\frac{315}{2}, \quad CY_{33}=0,
\]
Thus the directions $e_1+e_2$ and $e_1-e_2$ are eigenflag directions, that is they parametrize the possible directions for the gradients of potential LCWs. However,  they can not be realized by LCWs, since their orthogonal complements are not integrable (see \cite[Theorem 6]{AFG}): in the first case,
$(e_1+e_2)^\perp$ is spanned by $e_1-e_2$ and $e_3$, but
\[
[e_1-e_2,e_3]= 4e_2-6e_1
\]
which does not lie in the span of $e_1-e_2$ and $e_3$. The other case is similar.

\subsection{3-manifolds and metric rigidity for many limiting Carleman weights}

In this section we determine metrics that admit several limiting Carleman weights. The   example to keep in mind is the following:

\begin{example}
\label{ex:3dim}
Assume $M$ is a 3-dimensional Riemannian manifold admitting a limiting Carleman weight. Then $M$ is, up to a conformal factor, contained in 
$\R\times S$ for some surface $S$.
We claim that:
\begin{enumerate}
\item if $S$ is a surface of revolution different from the flat plane, then $M$ has two LCWs;
\item if $S$ is the flat plane, then $M$ is conformally flat.
\end{enumerate}

Recall that a \emph{surface of revolution} is a surface with a Riemannian metric admitting a Killing vector field.

\end{example}

\begin{proof}
The second statement is obvious. For the first, observe that if $S$ is a surface with a Killing field, its metric can be written as
\[
g= dt^2 + dr^2+f(r)^2 d\theta^2,
\]
with a nonvanishing $f$. Here $t$ is an LCW. Dividing by $f^2$, we observe that
\[
\dfrac{1}{f^2}g=d\theta^2+
\dfrac{1}{f^2}(dt^2+dr^2),
\]
and thus there is a second limiting Carleman weight in the $\theta$-direction.
%
%
%
\end{proof}


In what follows, we want to see what restrictions appear in the metric when there are several limiting Carleman weights. We start with the case when there are
 three or more limiting
Carleman weights, proving that this corresponds to the conformally flat situation.

\begin{thm}
Let $M$ be a 3-dimensional Riemannian manifold that admits three limiting Carleman weights whose gradients are pointwise linearly independent. Then $M$ is conformally flat.
\end{thm}

Notice that by a standard density argument, the Theorem also works in the case when the limiting Carleman weights are independent in a \emph{dense open set}.

\begin{proof}
We will show that $CY_p\equiv 0$ everywhere. To do this, we diagonalize the $2$-tensor $CY_p$. From Theorem \ref{Cotton2}, the associated matrix in an orthonormal base $\mathcal{B}=\{\,e_1,e_2,e_3\}$ with $e_1$ an eigenflag, has the form
\[
\begin{pmatrix}
0 & a & b \\
a & 0 & 0 \\
b & 0 & 0
\end{pmatrix}
\]
with $a,b\in\R$; its
 eigenvalues are $0$ and $\pm\sqrt{a^2+b^2}$.

 Suppose that $CY_p \not\equiv 0$. Then there is a unit vector $v \in e_1^{\perp}$ ($v$ is a multiple of $b e_2 - a e_3$) with $CY_p(v) = 0$. Similarly, if $e_1'$ and $e_1''$ are the other two eigenflag directions such that $\{ e_1, e_1', e_1''\}$ are linearly independent, there are unit vectors $v' \in (e_1')^{\perp}$ and $v'' \in (e_1'')^{\perp}$ with $CY_p(v') = CY_p(v'') = 0$. Now it is not possible that $v, v', v''$ are collinear (since no unit vector can be orthogonal to $e_1$, $e_1'$, $e_1''$), so it follows that $CY_p(w) = 0$ for all $w$ in some two-dimensional subspace on $T_p M$. Thus $0$ is an eigenvalue of multiplicity at least two, but since $CY_p$ is trace free all eigenvalues must be zero. Thus $CY_p \equiv 0$, which is a contradiction.
\end{proof}

Next we study the case with two LCWs. 

\begin{thm}
Let $M$ be a Riemannian 3-manifold admitting two limiting Carleman weights with linearly independent gradients. Suppose $M$ is not conformally flat.
Then $M$ is conformal to $\R \times S$, where $S$ is a nonflat surface of  revolution.
\end{thm}
\begin{proof}
Denote the two LCWs by $\varphi_1$ and $\varphi_2$. Then the gradients of $\varphi_1$ and $\varphi_2$ give rise to eigenflag directions by Theorem \ref{Cotton2}, and consequently $\det(CY) = 0$ by \cite[Corollary 1.7]{AFGR}. Since $M$ is not conformally flat, $CY$ is not null, and hence \cite[Lemma 14]{AFG} states that $M$ has exactly two eigenflag directions. Moreover, the proof of \cite[Lemma 14]{AFG} implies that these eigenflag directions are necessarily orthogonal. It thus follows that the gradients of $\varphi_1$ and $\varphi_2$ are orthogonal.

Since $\varphi_1$ and $\varphi_2$ are orthogonal LCWs, by Theorem \ref{thm:diagonal_metric} there is a chart $(x,y,z)$ where $x = \varphi_1$, $y = \varphi_2$, and the metric is conformal to 
\[
\begin{pmatrix}
1 & 0 & 0 \\
0 & a(z) & 0 \\
0 & 0 & b(z)
\end{pmatrix} .
\]
After a change of coordinates of the type $(x,y,z)\rightarrow (x,y,v(z))$, the metric is conformal to 
\[
\begin{pmatrix}
1 & 0 & 0 \\
0 & a(z) & 0 \\
0 & 0 & 1
\end{pmatrix} ,
\]
and thus $M$ is conformal to $\mR \times S$ for some surface of revolution $S$.
\end{proof}

\section{Limiting Carleman weights in 4-manifolds} \label{sec7}

We start this section by determining the local structure of metrics in 4-dimensional manifolds that admit several orthogonal LCWs. The case of four orthogonal LCWs is covered by Corollary \ref{cor:every orth LCW}, and corresponds to conformally flat metrics. This leaves the cases of two and three  LCWs, that we cover in the next two sections.

\subsection{Structure of a 4D manifold with two orthogonal LCWs.}
\label{subsec:4dim,2LCWs}
By Theorem \ref{thm:diagonal_metric}, the metric $g$ is written in some coordinates $(t,x,y,z)$, where $t$ and $x$ are LCWs, as a conformal multiple of 
\[
\begin{pmatrix}
1 & 0 & 0 & 0 \\
0 & a(y,z) & 0 & 0 \\
0 & 0 & b_{11}(y,z) & b_{12}(y,z) \\
0 & 0 & b_{21}(y,z) & b_{22}(y,z)
\end{pmatrix} .
\]

The above metric is the product
$
\R\times \tilde{M}^3,
$
where the metric $\tilde{g}$ taken in $\tilde{M}^3$ is a warped product over a surface $(S,\tilde{h})$ with fiber a line; i.e,
at each point $(x,y,z)$,
\[
\tilde{g}=\pi^*\tilde{h}+a(y,z) \,dx^2
\]
where $\pi:\R^3\to \R^2$ is the projection onto the $(y,z)$-coordinates and $\tilde{h}$ is given by the $2\times 2$ submatrix $(b_{ij})$.

It is natural to ask to what extent the structure for the metric resembles this if we only assume existence of two LCWs, \emph{not necessarily orthogonal}. This cannot happen for type B Weyl tensors, since eigenflags for that case are always orthogonal. But the case of type C Weyl tensors allows this possibility, although we have not been able to produce any example for this case.

\subsection{Structure of a 4D manifold with three orthogonal LCWs}
\label{subsec:4dim,3LCWs}
In this case, the metric is by Theorem \ref{thm:diagonal_metric} a conformal multiple of 
\[
\begin{pmatrix}
1 & 0 & 0 & 0 \\
0 & a(z) & 0 & 0 \\
0 & 0 & b(z) & 0 \\
0 & 0 & 0 & c(z)
\end{pmatrix} .
\]
If $h(z)$ is a primitive of $\sqrt{c(z)}$, we can define a new coordinate $\bar{z}=h(z)\cdot z$; using $(t,x,y,\bar{z})$ as the new system of coordinates, the above metric becomes 
\[
\begin{pmatrix}
1 & 0 & 0 & 0 \\
0 & a(\bar{z}) & 0 & 0 \\
0 & 0 & b(\bar{z}) & 0 \\
0 & 0 & 0 & 1
\end{pmatrix} .
\]
We can write this as the Riemannian product of $\R$ times a 3-dimensional manifold with a metric that is an iterated warped product; namely,
if $\pi_1(x,y,\bar{z})=(y,\bar{z})$, and $\pi_2(y,\bar{z})=\bar{z}$, then the 3-dimensional factor has the metric
\[
\bar{g}=
a(\bar{z})\,dx^2 +\pi_1^*(\,b(\bar{z})\,dy^2+\pi_2^*\,d\bar{z}^2\,).
\]

%
%

\subsection{Product 4-manifolds and LCWs.}
A 4-dimensional Riemannian manifold can split as a metric product only in two ways: either as $\R\times N^3$ (and in that case, it would have a LCW), or as a product of surfaces. In the latter case, the Weyl tensor is of type C, and as a consequence it has plenty of eigenflag directions.
The first three authors proved in \cite{AFG} that a product of two surfaces has a LCW if and only if one of the factors is a surface of revolution.
We now provide a simpler proof of Theorem 11 in \cite{AFG} (cf. Lemma 20 in \cite{AFG}).

 \begin{thm}
 A product of two surfaces has a LCW if and only if one of the factors is a surface of revolution.
 \end{thm}
\begin{proof}
As pointed in \cite[proof of Theorem 11]{AFG}, there is a LCW whenever one of the factors is a surface of revolution, so we need to prove only one implication.

Let $M=S_1\times S_2$ be a product of two surfaces with the metric $g=g_1\oplus g_2$, and denote by $\varphi$ a LCW for this metric. At any point, the gradient of $\varphi$ is an eigenflag direction, and since the Weyl tensor of a product metric is of type C,
 $\nabla\varphi$ must be tangent to one of the two factors, say $S_1$ (see \cite[Lemma 19]{AFG}).

Let $x,y$ be coordinates in $S_1$ and $z,t$ be isothermal coordinates in $S_2$, so that the metric in $S_2$ is

\[
\begin{pmatrix}
\lambda(z,t) & 0 \\
0 & \lambda(z,t)
\end{pmatrix} .
\]

We first remark that $\varphi$ does not depend on $z$ or $t$, since
\[
\partial_z\varphi = d\varphi(\partial_z) = \langle \nabla \varphi,\partial_z \rangle = 0,
\]
as $\nabla\varphi$ is everywhere tangent to $S_1$.
This means that we can consider $\varphi$ as a function of the first factor $S_1$ only.
We now construct a companion function $\psi:S_1\rightarrow\R$, whose differential vanishes on $\nabla\varphi$.
This is always possible by basic ODE theory (there are coordinates $(x_1,x_2)$ on $S_1$ with $\nabla \varphi = \partial_{x_1}$, and one can take $\psi = x_2$).

We write the metric of $M$ in the coordinates $\{x=\varphi,y=\psi,z,t\}$ (whose coordinate vector fields are denoted $\{\partial_1,\partial_2,\partial_3,\partial_4\}$):
\[
\begin{pmatrix}
a(x,y) & 0 & 0 & 0\\
0 & b(x,y) &0 & 0 \\
0 & 0 & \lambda(z,t) & 0\\
0 & 0 & 0 & \lambda(z,t)
\end{pmatrix} .
\]

In a multiple $f\cdot g$ of $g$, $\nabla \varphi$ is a unit parallel vector field, and the metric is written, in the same coordinates $\{\varphi,\psi,z,t\}$ (since $\partial_2,\partial_3,\partial_4$ are orthogonal to $\partial_1$ in $g$, they are also orthogonal in any multiple of $g$):
\[
\begin{pmatrix}
1 & 0 & 0 & 0\\
0 & h_2(y,z,t) &0 & 0 \\
0 & 0 & h_3(y,z,t) & 0\\
0 & 0 & 0 & h_4(y,z,t)
\end{pmatrix} .
\]
Thus
\[
f(x,y,z,t)\begin{pmatrix}
a(x,y) & 0 & 0 & 0\\
0 & b(x,y) &0 & 0 \\
0 & 0 & \lambda(z,t) & 0\\
0 & 0 & 0 & \lambda(z,t)
\end{pmatrix}
=
\begin{pmatrix}
1 & 0 & 0 & 0\\
0 & h_2(y,z,t) &0 & 0 \\
0 & 0 & h_3(y,z,t) & 0\\
0 & 0 & 0 & h_4(y,z,t)
\end{pmatrix} .
\]
Hence $f=\frac{1}{a}$ depends only on $x,y$, but $f=\frac{h_3}{\lambda}$ does not depend on $x$.
Hence $a=\frac{1}{f}$ and $b=\frac{h_2}{f}$ do not depend on $x$.
Hence $g$ is written as
\[
\begin{pmatrix}
a(y) & 0 & 0 & 0\\
0 & b(y) &0 & 0 \\
0 & 0 & \lambda(z,t) & 0\\
0 & 0 & 0 & \lambda(z,t)
\end{pmatrix}
\]
and $S_1$ is a surface of revolution.
\end{proof}

%
%

\section{Eigenflags vs. LCWs: the 4-dimensional case} \label{sec8}

In \cite{AFG}, the first three authors classified Weyl tensors of 4-manifolds depending on the number of existing eigenflags
(see Lemma \ref{lem:Weyl}).
Some examples in that paper showed that eigenflags may not necessarily be realized by LCWs.
It was not clear at that moment if there were examples with eigenflag directions but no LCWs at all.

In this section we provide examples of  Riemannian 4-manifolds, one with a Weyl tensor of type B, one with a Weyl tensor of type C, such that no eigenflag direction comes from a LCW. The latter case could be considered as the most difficult, since the eigenflag directions fill the union of two 2-planes in each tangent space of $M$ and we need to rule out the existence of LCWs for all of them.

\subsection{General strategy}
Both examples are constructed using left invariant metrics on 4-dimensional Lie groups. Denote by $\mathcal{B}=\{e_0,e_1,e_2,e_3\}$ a basis of left invariant vector fields. For each example, we will need to construct the structure constants of the group, that determine its algebraic structure. These are constants $c_{ij}^k$ for $0\leq i,j,k\leq 3$ such that
\[
[e_i,e_j]= \sum_k c_{ij}^k e_k, \qquad c_{ij}^k\in\mathbb{R}.
\]

The metric $g$ is then defined by asking that $\{e_0,e_1,e_2,e_3\}$ forms an orthonormal basis at every point of the group.

The curvature tensor of $M$ is computed with the formulas of \cite{Mil}; the computations could be done by hand, or be easily implemented using some mathematical software system as \texttt{SAGE} \cite{Sage}. Although \cite{Mil} only provides formulas for the curvature and Ricci tensor components in the base $\mathcal{B}$, we can compute the  Weyl tensor by calculating first the Schouten tensor
\begin{equation}\label{Sformula}
S=\frac{1}{n-2}\left( Ric-\frac{1}{2(n-1)}s g\right)
\end{equation}
and then using that

\begin{equation}\label{decomposition}
R= W + S \owedge g\end{equation}
where $\owedge$ is the well-known Kulkarni-Nomizu product of symmetric $2$-tensors  defined as

\[ (\alpha \owedge \beta)_{ijkl}=\alpha_{ik}\beta_{jl}+\alpha_{jl}\beta_{ik}-\alpha_{il}\beta_{jk}-\alpha_{jk}\beta_{il} \]
where $R$ and $W$ are understood as $(0,4)$ tensors.

Once we compute $W$, we can compute whether it has eigenflags, and the type of the tensor.

To check whether the metric has LCWs along these eigenflags, we will
 use \cite[Theorem 1.5]{AFG} that we recall here: if a vector field $e$ of eigenflags can be realized with a LCW, then its orthogonal distribution $D=e^\perp$ is integrable and umbillical:
\begin{itemize}
\item \emph{Integrable:} Define the second fundamental form of the distribution $D$ as
$$
II(X,Y) = P_{D^\perp}(\nabla_X Y)
$$
where $P_{D^\perp}$ is the projection onto $D^\perp$ and $X$, $Y$, are vector fields tangent to $D$; the distribution is integrable if and only this form is symmetric.
\item \emph{Umbilical:} there exists a vector field $H\in D^\perp$, called the \emph{mean curvature vector field of $D$}, such that for $X,Y \in D$ and $Z \in D^{\perp}$ it holds that
\begin{equation}
\label{eq:umbilic definition}
g( \nabla_X Y, Z ) =  g( X, Y )g(Z, H)
\end{equation}
Umbilicity implies integrable.
\end{itemize}

\subsection{A 4-dimensional metric with a type B Weyl tensor and no LCWs}
\label{subsec:4dim-noLCW}
The structure constants all vanish, except for
\begin{align*}
[e_0,e_1]&=-\frac{1}{2}e_2, \quad
[e_0,e_2]=-e_1,  \quad
\\
[e_1,e_2]&=e_0, \quad
[e_1,e_3]=-\frac{1}{2}e_2, \quad
[e_2,e_3]=-e_1.
\end{align*}

The Levi-Civita connection is determined by the above and the values
\begin{align*}
\nabla_{e_0}{e_0}&=0, \quad
\nabla_{e_0}{e_1}=-\frac{1}{4}e_2, \quad
\nabla_{e_0}{e_2}=\frac{1}{4}e_1, \quad
\nabla_{e_0}{e_3}=0,
\\
\nabla_{e_1}{e_1}&=0, \quad
\nabla_{e_1}{e_2}=-\frac{1}{4}e_0+\frac{3}{4}e_3,\quad
\nabla_{e_1}{e_3}=-\frac{3}{4}e_2,
\\
\nabla_{e_2}{e_2}&=0, \quad
\nabla_{e_2}{e_3}=-\frac{3}{4}e_1, \quad
\nabla_{e_3}{e_3}=0,
\end{align*}
that arise from applying the formulas in \cite{Mil}.

In the basis $\mathcal{B}$, the  components of the curvature tensor are given as
\begin{align*}
R_{ijk\ell}=g(R(e_i,e_j)e_k,e_\ell),
\end{align*}
Due to the symmetries of the curvature, we have that
\[
R_{ijk\ell}=-R_{jik\ell}=-R_{ij\ell k}=R_{k\ell ij}.
\]

After carrying the computations, we get that (except for the above index permutations), the nonzero components f the curvature tensor are
\begin{align*}
R_{0101}=-\frac{11}{16}, \quad
R_{0113}=-\frac{9}{16}, \quad
R_{0202}=\frac{1}{16}, \quad
R_{0223}=-\frac{9}{16}
\\
R_{1212}=\frac{5}{8}, \quad
R_{1313}=-\frac{3}{16}, \quad
R_{2302}=-\frac{9}{16}, \quad
R_{2323}=-\frac{15}{16}.
\end{align*}

The nonzero Ricci tensor components are
\begin{align*}
\Ric_{00}=-\frac{5}{8}, \quad
\Ric_{03}=\frac{9}{8}, \quad
\Ric_{11}=-\frac{1}{4},
\\
\Ric_{22}=-\frac{1}{4}, \quad
\Ric_{30}=\frac{9}{8}, \quad
\Ric_{33}=-\frac{9}{8}.
\end{align*}
Observe that the metric is not Einstein due to the presence of a nonzero term off the diagonal.

The scalar curvature is the trace of the Ricci tensor, resulting on
\[
s=-\frac{9}{4}.
\]

The Schouten tensor can be now computed using formula \eqref{Sformula} to get
\begin{align*}
S_{00}=-\frac{1}{8}, \quad
S_{03}=\frac{9}{16}, \quad
S_{11}=\frac{1}{16}
\\
S_{22}=\frac{1}{16}, \quad
S_{30}=\frac{9}{16}, \quad
S_{33}=-\frac{3}{8}
\end{align*}

Finally, the nonzero Weyl tensor components are
\begin{align*}
W_{0220}&= W_{1331} = W_{2002} = W_{3113}=-\frac{1}{8}, \quad
W_{0330}= W_{1221} = W_{2112} = W_{3003}=-\frac{1}{2}
\\
W_{0202}&= W_{1313} = W_{2020} = W_{3131}=\frac{1}{8}, \quad
W_{0303} = W_{1212} = W_{2121} = W_{3030}=\frac{1}{2}
\\
W_{0110}&= W_{1001} = W_{2332} = W_{3223}=\frac{5}{8}, \quad
W_{0101} = W_{1010} = W_{2323} = W_{3232}=-\frac{5}{8},
\end{align*}
where we are seeing $W$ as a $(0,4)$-tensor. To write it as an endomorphism in bivectors, we start by taking a basis of these given by
\[
\mathcal{B}'=\{\,e_0\wedge e_1 , e_0\wedge e_2 , e_0\wedge e_3 , e_1\wedge e_2 , e_1\wedge e_3 , e_2\wedge e_3  \,\},
\]
and obtain
\[
W=
\begin{pmatrix}
-\frac{5}{8} &  &  &  &  &  \\
 & \frac{1}{8} &  &  &  &  \\
 &  & \frac{1}{2} &  &  &  \\
 &  &  & \frac{1}{2} &  &  \\
 &  &  &  & \frac{1}{8} &  \\
 &  &  &  &  & -\frac{5}{8}
\end{pmatrix}
 \]

Thus $W$ diagonalizes in simple bivectors obtained from the canonical basis.
Observe also that, since there are three different eigenvalues, the Weyl tensor is of type B and not C.

The eigenflags for this metric are precisely
the elements of the canonical basis $\mathcal{B}$.
To rule out that they could correspond to LCWs, we compute the second fundamental forms of the distributions orthogonal to each one of the vector fields $e_i$:

\begin{itemize}
\item $D_0=e_0^\perp$: in the basis $e_1,e_2,e_3$ the second fundamental form is
\[
\begin{pmatrix}
0 & \frac{1}{4} & 0 \\
\frac{5}{4} & 0 & 0 \\
0 & 0 & 0
\end{pmatrix};
\]
\item $D_1=e_1^\perp$: in the basis $e_0,e_2,e_3$ the second fundamental form is
\[
\begin{pmatrix}
0 & -\frac{1}{4} & 0 \\
-\frac{5}{4} & 0 & \frac{3}{4} \\
0 & -\frac{1}{4} & 0
\end{pmatrix};
\]
\item $D_2=e_2^\perp$: in the basis $e_0,e_1,e_3$ the second fundamental form is
\[
\begin{pmatrix}
0 & \frac{1}{4} & 0 \\
-\frac{1}{4} & 0 & \frac{3}{4} \\
0 & \frac{1}{4} & 0
\end{pmatrix};
\]
\item $D_3=e_3^\perp$: in the basis $e_0,e_1,e_2$ the second fundamental form is
\[
\begin{pmatrix}
0 & 0 & 0 \\
0 & 0 & -\frac{3}{4} \\
0 & -\frac{3}{4} & 0
\end{pmatrix}.
\]
\end{itemize}

Integrability means that the second fundamental form should be symmetrical; this already rules out $e_0,e_1,e_2$ as tangent to possible LCWs.

Umbilicity implies that the second fundamental form should be a multiple of the identity; this rules out the left possibility $e_3$.

\subsection{A 4-dimensional metric that is not a product, has Weyl tensor of type C, and no LCWs}
\label{subsec:4dim_no_LCWs}
\mbox{}
\\
 Denote by $\mathcal{B}=\{e_0,e_1,e_2,e_3\}$ a basis of left invariant vector fields. The group structure arises from
\begin{align*}
[e_0,e_1]&=-e_2-e_3, \quad
[e_0,e_2]=-e_1+e_3,  \quad
[e_0,e_3]=-e_1-e_2,
\\
[e_1,e_2]&=e_0+e_3, \quad
[e_1,e_3]=e_0-3e_2, \quad
[e_2,e_3]=-e_0-3e_1.
\end{align*}

The metric is defined by asking that $\{e_0,e_1,e_2,e_3\}$ forms an orthonormal basis.

  Once the computations are carried out, the nonzero components of $W$ on the basis $\mathcal{B}$
are (modulo trivial rearrangements of the indices)
\[
W_{0101}=W_{2323}=-8;
W_{0202}=W_{0303}=W_{1212}=W_{1313}=4.
\]
This means that the planes $\pi_1=e_0\oplus e_1$ and $\pi_2=e_2\oplus e_3$ are entirely formed by eigenflags, thus giving a Weyl tensor of type C.

To rule out the existence of LCWs, we will use \cite[Lemma 15]{AFG}. Suppose that some linear combination (where $\alpha:M\to\R$ is some function)
\[
X=\cos(\alpha)e_0+\sin(\alpha)e_1
\]
was a LCW. Then its orthogonal distribution
(that is spanned by $Y=-\sin(\alpha)e_0+\cos(\alpha)e_1$, $e_2$ and $e_3$) would have to be umbilical, i.e, integrable and such that the second fundamental form $B$ of its integral leaves is a multiple of the identity. But
\[
B(e_3, Y)=g(\nabla_{e_3} Y, X)=\frac{1}{2},
\]
instead of vanishing.

An entirely similar argument shows that there are also no LCWs tangent to the planes spanned by $e_2$ and $e_3$.

To rule out the possibility that the metric were a Riemannian product, we
start by observing that there can not be one dimensional factors, since these are equivalent to LCWs, thus we only need to consider product of surfaces $S_1\times S_2$.

But in that case, the structure of the Weyl tensor implies that the 2-dimensional factors should be tangent at each point to the planes $\pi_1=e_0\oplus e_1$ and $\pi_2=e_2\oplus e_3$. Thus, it is enough to check that such distributions are not umbillical. A simple calculation shows that
\[
g(\nabla_{e_0}e_1, e_2)=\frac{1}{4}, \qquad
g(\nabla_{e_1}e_0, e_2)=-\frac{1}{4},
\]
thus proving that the second fundamental form of the distribution $e_0\oplus e_1$ is not symmetrical.

\end{document}